\documentclass[11pt,a4paper,twoside]{article}

\usepackage[utf8]{inputenc}
\usepackage[english]{babel}
\usepackage{amsthm,amsmath,amssymb, mathtools,amsfonts}  

\usepackage{epic}
\usepackage{graphicx,epstopdf,overpic,tikz}   
\usepackage{caption}
\usepackage{listings}
\usepackage{scrpage2}
\usepackage{paralist} 
\usepackage{multirow}
\usepackage{color}

\author{Tobias B\"ohle and Christian Kuehn}
\title{Mathematical Analysis of Nonlocal PDEs for Network Generation}

\newcommand{\norm}[1]{\left\lVert #1 \right\rVert } 

\newcommand{\be}{\begin{equation}}
\newcommand{\ee}{\end{equation}}
\newcommand{\benn}{\begin{equation*}}
\newcommand{\eenn}{\end{equation*}}
\newcommand{\bea}{\begin{eqnarray}}
\newcommand{\eea}{\end{eqnarray}}
\newcommand{\beann}{\begin{eqnarray*}}
\newcommand{\eeann}{\end{eqnarray*}}

\newcommand{\R}{\mathbb{R}} 
\newcommand{\N}{\mathbb{N}} 

\def\cD{\mathcal{D}}

\def\cF{\mathcal{F}}

\theoremstyle{definition}
\newtheorem{definition}{Definition}[section]		
\newtheorem{problem}[definition]{Problem}			

\theoremstyle{plain} 	
\newtheorem{lemma}[definition]{Lemma}				
\newtheorem{theorem}[definition]{Theorem}			
\newtheorem{proposition}[definition]{Proposition}	

\theoremstyle{remark}
\newtheorem{remark}[definition]{Remark}				

\theoremstyle{plain}

\numberwithin{equation}{section}

\begin{document}


\begin{center}
	{\Huge \bfseries Mathematical Analysis of Nonlocal PDEs for Network Generation}\\[3mm]
	Tobias B\"ohle\textsuperscript{1} and Christian Kuehn\textsuperscript{1}\\
	{\small \textsuperscript 1Technical University of Munich, Faculty of Mathematics, Boltzmannstr.~3, 85748 Garching b.~M\"unchen, Germany}
\end{center}

\hrulefill
\paragraph{Abstract} 
In this paper, we study a certain class of nonlocal partial differential equations (PDEs). The equations arise from a key problem in network science, i.e., network generation from local interaction rules, which result in a change of the degree distribution as time progresses. The evolution of the generating function of this degree distribution can be described by a nonlocal PDE. To address this equation we will rigorously convert it into a local first order PDE. Then, we use theory of characteristics to prove solvability and regularity of the solution. Next, we investigate the existence of steady states of the PDE. We show that this problem reduces to an implicit ODE, which we subsequently analyze. Finally, we perform numerical simulations, which show stability of the steady states. 
\\

\hrulefill

\pagenumbering{arabic}
\pagestyle{headings}

\section{Introduction}

In this paper we will study a class of PDEs, which have the structure
\begin{equation}
\label{eq:introeq}
	G_t(x,t) = \cF(G_x(x,t),G(x,t),x;G_x(1,t)),
\end{equation}
where $G=G(x,t)$ is the unknown, subscripts denote partial derivatives, $(x,t)\in\cD$ for an open domain $\cD\subset \R\times \R_{\ge 0}$ containing the line $\{1\} \times \R_{> 0}$, and the mapping $\cF$ as well as the boundary/initial conditions will be specified below. In~\eqref{eq:introeq}, the nonlocal term is given by $G_x(1,t)$, i.e., the right-hand side of the evolution PDE depends nonlocally on a single spatial point located at $x=1$.

The PDE~\eqref{eq:introeq} arises naturally from modelling dynamics of networks/graphs~\cite{Silk2016}. Let us briefly motivate, why it is of key importance to study the dynamics of networks, respectively network generating mechanisms. The most classical model of complex networks is the Erd\H{o}s-Reny\'i model~\cite{ErdosRenyi}, where each edge between a fixed number of $N$ nodes is present with equal probability $p$. We can think of the Erd\H{o}s-Reny\'i model as a dynamical generation process. We start with a graph with $N$ vertices and no edges. At each discrete time step, we look at a new pair of vertices and with probability $p/N$ we add an edge between them. After this process has finished, the probability of a vertex having degree $k$ converges as $N\to \infty$ (weakly) to a Poisson distribution. However, it is by now understood~\cite{Albert2002} that many realistic complex networks~\cite{Newman} are not well described by the Erd\H{o}s-Reny\'i model. In particular, the degree distribution rarely obeys a Poisson distribution but seems to resemble more closely a power law; see e.g.~\cite{Broder2000}. Of course, one always has to be careful with postulating exact power laws~\cite{StumpfPorter}. 

In summary, it is certainly useful to study different theoretical mechanisms for network generation, and then observe the resulting degree distributions. This motivated the study by Silk et al~\cite{Silk2016}, which is our starting point. As an example of a network generation process, consider preferential attachment of edges\footnote{The preferential attachment that we here refer to does not correspond to the usual procedure of creating power law graphs. In particular, graphs of the Barab\'asi-Albert model are formed by subsequently introducing new vertices and then connecting them to the rest of the network~\cite{BarabasiAlbert}. However, when speaking of preferential attachment, we do not introduce new nodes to the network but only new edges between already existing nodes. The process to create Barab\'asi-Albert graphs is called addition of nodes by preferential attachment, see Table \ref{tab:processoverview}.}: Randomly pick two unconnected vertices with a probability proportional to their degree and create an edge between them. Call the rate at which this process of preferential attachment of edges takes place $l_p\geq 0$. 

Given a degree distribution $p_k$ of our graph, we now want to investigate how this process affects the degree distribution~\cite{Silk2016}. Therefore, we first rephrase this process under the additional assumption that for each $t>0$ two nodes chosen with a probability proportional to their degree are almost surely disconnected. This allows us to put the process in the following words: Independently pick two nodes from the network with a probability proportional to their degree and increase their degree by one, i.e. add an edge between them. As the probability of picking a vertex with degree $k$ is proportional to its degree, we set $q_k = c k p_k$, where $c$ is chosen such that $\sum_{k=0}^\infty q_k=1$. By the law of total probability, the total change of $p_k$ is the sum of the effect of degree-$l$-nodes on $p_k$ multiplied with the probability that we choose a node of degree $l$. That is,
\begin{align*}
	\frac{\mathrm dp_k}{\mathrm d\tilde t} = \sum_{l = 0}^\infty \frac{\mathrm d p_{k,l}}{\mathrm d \tilde t} q_l,
\end{align*}
where $\tilde t = 2l_p t$ and $\frac{\mathrm d p_{k,l}}{\mathrm d \tilde t}$ denotes the change of $p_k$ if an $l$-degree node increases its degree by one. Now it remains to evaluate $\frac{\mathrm d p_{k,l}}{\mathrm d \tilde t}$. For $l = k-1$, a node changes its degree from $k-1$ to $k$, so $p_k$ increases by one normalized unit, i.e. we set $\frac{\mathrm d p_{k,k-1}}{\mathrm d \tilde t} = 1$. For $k = l$, the degree of a node having degree $k$ will increase to $k+1$. Thus, it is no longer $k$ and consequently $p_k$ reduces by one normalized unit, so we set $\frac{\mathrm d p_{k,k}}{\mathrm d \tilde t} = -1$. In case that neither $l = k-1$ nor $l = k$, the probability $p_k$ remains unaffected, so $\frac{\mathrm d p_{k,l}}{\mathrm d \tilde t} = 0$, for $l\notin\{k,k-1\}$. Thus, the final ODE for this process is
\begin{align}
	\frac{\mathrm d p_k}{\mathrm dt} = 2l_p \frac{1}{\sum_{n=0}^\infty n p_n}\left((k-1)p_{k-1}-kp_k\right),
\end{align}
in the infinite network limit $N\to \infty$, where we set $p_{-1}\equiv 0$. There are many other processes one can now consider, e.g., re-wiring of edges, deletion of vertices and/or edges, random additions of vertices, etc. Combining a number of these processes, which are all listed in Table \ref{tab:processoverview}, Silk et al~\cite{Silk2016} find the final equations for $p_k$ as
\begin{subequations}
\label{eq:infiniteODEs}
	\begin{align}
		\label{eq:randomrewiring}		
		\frac{\mathrm d p_k}{\mathrm dt} &= \omega_r 
		\left[(k+1)p_{k+1}-kp_k+\left(\sum_n np_n\right)(p_{k-1}-p_k)\right]\\
		\label{eq:preferentialrewiring}
		&\quad+\omega_p[(k+1)p_{k+1}-kp_k+(k-1)p_{k-1}-kp_k]\\
		\label{eq:deletionoflinks}
		&\quad+l_d[(k+1)p_{k+1}-kp_k]\\
		\label{eq:randomadditionoflinks}
		&\quad+2l_r[p_{k-1}-p_k]\\
		\label{eq:preferantialadditionoflinks}
		&\quad+2l_p\left[\frac{1}{\sum_n np_n}((k-1)p_{k-1}-kp_k)\right]\\
		\label{eq:deletionofnodes}
		&\quad+n_d\left(\sum_n np_n\right)[(k+1)p_{k+1}-kp_k]\\
		\label{eq:randomadditionofnodes}
		&\quad+n_r[m(p_{k-1}-p_k)-p_k+\delta_{m,k}]\\
		\label{eq:additionofnodesbypreferentialattachment}
		&\quad+n_p\left[\frac{m}{\sum_n np_n}((k-1)p_{k-1}-kp_k)-p_k+\delta_{m,k}\right].
	\end{align}
\end{subequations}
Here, each line represents an individual process\footnote{However, some of the processes may need additional normalization. In particular, we work with the equations derived in~\cite{Silk2016} as our starting point.}. The coefficients $\omega_r,\omega_p,l_d,l_r,l_p,n_d,n_r,n_p$ are nonnegative real constants, which display the rate at which the corresponding processes take place. Furthermore, $m\in \N_0$ is also a nonnegative constant which may only take values in the integers. In the above equation we set $p_{-1} \equiv 0$.

\begin{table}[h]
\centering

\begin{tabular}{|l|l|}
	\hline
	Process & Short Description\\ \hline
	
	Random rewiring \eqref{eq:randomrewiring} & Select a link at random, break it,\\
	&and connect one of the nodes \\
	&	to another uniformly chosen random node.\\ \hline
	
	Preferential rewiring \eqref{eq:preferentialrewiring} & Same as random rewiring, except that  \\
	&the last node is chosen with \\
	&a probability proportional to its degree.\\ \hline
	
	Deletion of links \eqref{eq:deletionoflinks} & A randomly selected link \\
	&is deleted from the network.\\ \hline
	
	Random addition of links \eqref{eq:randomadditionoflinks} & Two unconnected nodes are picked randomly \\
	&and an edge between them is introduced.\\ \hline	
	
	Preferential addition of links \eqref{eq:preferantialadditionoflinks} & Same as random addition of links, \\
	&except that the two nodes are picked with \\
	&a probability proportional to their degree.\\ \hline
		
	Deletion of nodes \eqref{eq:deletionofnodes} & Select a node at random and delete it \\
	&from the network including all its edges.\\ \hline
		
	Random addition of nodes \eqref{eq:randomadditionofnodes} & Introduce a new node of degree $m$ \\
	&into the network and choose \\
	&its neighbors randomly from the network.\\ \hline
		
	Addition of nodes & Same as random addition of nodes, except that \\
	by preferential attachment \eqref{eq:additionofnodesbypreferentialattachment} &the new neighbors are chosen with \\
	&a probability proportional to their degree.\\ \hline
\end{tabular}
\caption{\label{tab:processoverview}Processes included in \eqref{eq:infiniteODEs}. For a detailed description of the processes, see \cite{Silk2016}.}
\end{table}

By using the generating function $G(x,t) = \sum_{k=0}^\infty p_k(t) x^k$
one can transform this infinite system of ODEs into one single PDE~\cite{Silk2016}, which has the structure~\eqref{eq:introeq}. We multiply \eqref{eq:infiniteODEs} with $x^k$ and then sum over $k\ge 0$. The resulting PDE describing the evolution of $G(x,t)$ is then given by \eqref{eq:MainEquation}. Together with an initial condition $G(x,0) = \sum_{k=0}^\infty \tilde p_k x^k$, we can then formulate this as a mathematical problem as follows:

\begin{problem}\label{prob:main}
Let $\cD\subset \R\times \R_{\ge 0}$ be given with $\cD$ being an open set containing the line $\{1\}\times \R_{> 0}$ in its interior $\mathring \cD$. Find a function $G\colon  \cD\to \R$ such that
\begin{align}
\begin{split}\label{eq:MainEquation}
			G_t(x,t) &= (x-1)\left[x\left( \omega_p+\frac{2l_p+n_p m}{G_x(1,t)}\right)
			-\omega_r-\omega_p-l_d-n_dG_x(1,t)\right]G_x(x,t)\\ &+\left[ (x-1)
			\left(\omega_rG_x(1,t)+2l_r+n_rm\right)-n_r-n_p\right] G(x,t) +(n_r+n_p)x^m
\end{split}
\end{align}
for all $(x,t)\in \cD$ and additionally
\begin{align}	
G(x,0) &= \sum_{k=0}^\infty \tilde p_k x^k \eqqcolon \tilde h(x)\label{eq:InitialCond}
\end{align}
for all $(x,0)\in \partial \cD$, where $\tilde p_k \in[0,1],
\ \sum\limits_{k=0}^\infty \tilde p_k = 1$.
\end{problem}

In this paper, we provide a mathematical study of the nonlocal PDE~\eqref{eq:MainEquation}-\eqref{eq:InitialCond}. We prove the existence and regularity of (classical) solutions using an auxiliary problem in combination with the method of characteristics. In this context, we also characterize the domain $\cD$ and relate it to the convergence radius of the series defining the generating function $G(x,0)$. Furthermore, we fully characterize the existence of steady states analytically, and study their stability numerically. We observe that the steady states are globally stable under reasonable conditions.\medskip

\textbf{Acknowledgments:} CK would like to thank the VolkswagenStiftung for support via a Lichtenberg Professorship. CK also would like to thank Franz Achleitner for interesting discussions regarding nonlocal PDEs. TB and CK would like to thank an anonymous referee, whose comments have helped to improve the presentation of the results.

\section{Solution Theory}\label{sec:characteristics}

At first, is seems difficult to find a solution for Problem \ref{prob:main}, mainly because of the nonlocal term $G_x(1,t)$. However, the PDE \eqref{eq:MainEquation} is only of first order. If the nonlocal term was not there, we could try to apply the method of characteristics, a general method for dealing with local first order PDEs. In order to still be able to apply this method, we will first prove that the nonlocal PDE \eqref{eq:MainEquation} can be converted into a local first order PDE.

\subsection{Problem Equivalence}

\begin{proposition}\label{prop:DiffGlAnf}
	Let $G\in C^2(\cD)$ be a solution of Problem \ref{prob:main} with $\{1\} \times \R_{>0} \subset \mathring \cD$. Then, $G(1,t)=1$ for all $t\ge 0$ and $g(t) \coloneqq G_x(1,t)$ satisfies the initial value problem
	\begin{align}\label{eq:diffeqg}
		\left.
		\begin{array}{rl} g'(t) & = -n_d (g(t))^2-bg(t)+c \\
		g(0) &= \tilde h'(1)\\ \end{array}
		\right\}
	\end{align}
	with $b:=l_d+n_p+n_r\ge 0$ and $c:=2(l_p+l_r+m(n_p+n_r))\ge 0$.
\end{proposition}

\begin{proof}
If we insert $x=1$ into~\eqref{eq:MainEquation}, we immediately obtain the differential equation
\begin{align*}
G_t(1,t)= n_r+n_p-(n_r+n_p)G(1,t),
\end{align*}
describing the value of $G$ at $x=1$. Because the initial value is given by $G(1,0) = \tilde h(1) =1$, the unique solution of this differential equation is given by $G(1,t)=1$ for all $t\ge 0$. In order to show the differential equation for $g$ we differentiate~\eqref{eq:MainEquation} once with respect to $x$, and then set $x=1$. Using $\frac{\mathrm d}{\mathrm dx} G_{t} = \frac{\mathrm d}{\mathrm dt} G_{x}$, we infer the differential equation for $g$:
\begin{align*}
		\frac{\mathrm d}{\mathrm dt} g(t) &= \left[\left( \omega_p+\frac{2l_p+n_p m}{g(t)}\right)-\omega_r-\omega_p-l_d-n_dg(t)\right]g(t)\\ 		
		&\quad +\left(\omega_rg(t)+2l_r+n_rm\right)+\left[-n_r-n_p\right] g(t)+(n_r+n_p)m\\
		&=2l_p+n_pm+(2l_r+n_rm)+n_rm+n_pm\\
		&\quad -\left(\omega_r+l_d -\omega_r+n_r+n_p \right)g(t)-n_d(g(t))^2\\
		&=-n_d(g(t))^2-b g(t)+c,
\end{align*}
where we have used that $G(1,t) = 1$ for all $t\ge 0$. This is the desired equation.
\end{proof}

For the following analysis it helps to simplify the notation:

\begin{definition}
\label{def:H}
Let $g\in C^k(\R_{\ge 0})$ be a positive function. We define the function $H\colon \R^3\times \R_{\ge 0}\to \R$ by
\begin{align*}
		H(a,b,c,d) &\coloneqq (c-1)\left[c\left(\omega_p+\frac{2l_p+n_pm}{g(d)}\right)-\omega_r-\omega_p-l_d-n_dg(d)\right]a\\
		&+\left( (c-1)(\omega_r g(d)+2l_r+n_rm)-n_r-n_p\right) b +(n_r+n_p)c^m
	\end{align*}
\end{definition}

From now on we assume that $g\in C^k(\R_{\ge 0})$ for any $k\ge 2$ throughout this section. We now face the following problem:

\begin{problem}\label{prob:H}
We look for functions $G\colon \cD\subset \R\times \R_{\ge 0}\to \R$ with $\{ 1\}\times \R_{>0} \subset \mathring \cD$ such that
	\begin{align}
		\label{eq:HG}
		G_t = H(G_x,G,x,t) & \quad \text{for }(x,t)\in \cD\\
		\label{eq:HGini}		
		G(x,0)= \sum\limits_{k=0}^\infty p_k x^k \eqqcolon h(x) & \quad \text{for }(x,0)\in \cD
	\end{align}
	where $p_k\in [0,1],\ \sum\limits_{k=0}^\infty p_k=1$.
\end{problem}

\begin{lemma}[Problem Equivalence]\label{lem:probequiv}
	Suppose $\{1\}\times \R_{>0}\subset \mathring \cD$ and $\tilde h = h$. Then a function $G\in C^2(\cD)$ is a solution of Problem \ref{prob:main} if and only if it is a solution of Problem \ref{prob:H} with $g$ satisfying the initial value problem \eqref{eq:diffeqg}.
\end{lemma}

\begin{proof}
Suppose first that $G\in C^2(\cD)$ is a solution of Problem \ref{prob:main}. Then, by Proposition \ref{prop:DiffGlAnf}, $g$ satisfies the initial value problem \eqref{eq:diffeqg}. After substitution of $g(t) = G_x(1,t)$ into \eqref{eq:MainEquation} we see that $G$ is a solution of Problem \ref{prob:H}. Now suppose that $G\in C^2(\cD)$ is a solution of Problem \ref{prob:H} with $g$ satisfying the initial value problem \eqref{eq:diffeqg}. In order to show that $G$ is a solution to Problem \ref{prob:main} we basically have to carry out the same calculations as in the proof of Proposition \ref{prop:DiffGlAnf}. First we set $x=1$ in \eqref{eq:HG}. This gives us
	\begin{align*}
		G_t(1,t) = n_r+n_p-(n_r+n_p)G(1,t).
	\end{align*}
	Since $G(1,0) = h(1) = 1$ it follows that $G(1,t)=1$ for all $t\ge 0$. Now differentiating \eqref{eq:HG} with respect to $x$, then setting $x=1$ and using the previous result $G(1,t)=1$ together with the fact that $\frac{\mathrm d}{\mathrm dt}G_x = \frac{\mathrm d }{\mathrm dx}G_t$ yields
	\begin{align}\label{eq:diffeqGx1}
		\begin{split}
		\frac{\mathrm d}{\mathrm dt} G_{x}(1,t) &= \frac{2l_p+n_pm}{g(t)}G_x(1,t)+2l_r+m(2n_r+n_p)\\
		&+G_x(1,t)(-\omega_r-l_d-n_r-n_p)+\omega_r g(t)-n_d g(t) G_x(1,t).
		\end{split}
	\end{align}
Differentiating the initial condition \eqref{eq:HGini} and fixing $x=1$ imposes the initial value $G_x(1,0) = h'(1)$. A short calculation shows that $G_x(1,t) = g(t)$ is a solution to this differential equation because $g$ satisfies \eqref{eq:diffeqg}. As the solution to \eqref{eq:diffeqGx1} for a given initial value is unique, we can infer that $g(t) = G_x(1,t)$ for all $t\ge 0$. Substituting that into \eqref{eq:HG} and using Definition \ref{def:H} yields that $G$ is a solution to Problem \ref{prob:main}.
\end{proof}

In contrast to the nonlocal Problem \ref{prob:main}, we now face the local first order Problem \ref{prob:H}. In order to solve it, we make use of the method of characteristics~\cite[Section 3.2]{Evans2010}.

\subsection{Applying the Method of Characteristics}

To apply the method of characteristics to Problem \ref{prob:H}, we define
\begin{align*}
	z(t)\coloneqq G(x(t),t),\quad  p^1(t)\coloneqq G_x(x(t),t),\quad p^2(t) \coloneqq G_t(x(t),t).
\end{align*}
The general characteristic equations in \cite[Section 3.2]{Evans2010} then turn into 
\begin{align}
	\dot x(t) &= -H_a,\\
	\dot p^1(t) &= H_c+H_bp^1(t),\\
	\dot p^2(t) &= H_d+H_b p^2(t),\\
	\dot z(t) &= -H_ap^1(t)+p^2(t),
\end{align}
where the partial derivatives of $H=H(a,b,c,d)$ are always evaluated at the point $(p^1(t), z(t),x(t),t)$. If we now substitute the function $H$ defined in Definition \ref{def:H} and its derivatives into the above equations, we obtain
\begin{subequations}
\label{eq:char}
\begin{align}
	\label{eq:xchar}
	\dot x(t) 	=& -(x(t)-1)\left( -l_d-\omega_p-\omega_r-n_d g(t)+\left(\omega_p+\frac{2l_p+mn_p}{g(t)}\right)x(t) \right), \\
	\label{eq:p1char}
	\begin{split}
	\dot p^1(t) =& p^1(t) \left(x(t) \left(\frac{2 l_p+mn_p}{g(t)}+\omega_p\right)-n_d g(t)-l_d-\omega_p-\omega_r\right)+m (n_p+n_r)x(t)^{m-1}\\ &+p^1(t) (x(t)-1)\left(\frac{2 l_p+mn_p}{g(t)}+\omega_p\right)+z(t) (\omega_r g(t)+2 l_r+m n_r)\\ & -\left(n_p+n_r-(2l_r+mn_r+\omega_r g(t))(x(t)-1)\right) p^1(t),
	\end{split}\\
	\label{eq:p2char}
	\begin{split}
	\dot p^2(t) =& -\frac{(x(t)-1)g'(t)}{g(t)^2}\bigg(p^1(t)x(t)(2l_p+mn_p)+(n_d p^1(t)-\omega_r z(t))g(t)^2\bigg)\\
				&-\left(n_p+n_r-(2l_r+mn_r+\omega_r g(t))(x(t)-1)\right)p^2(t),
	\end{split}\\
	\label{eq:zchar}
	\begin{split}
	\dot z(t) 	=& \left(  x(t)(2l_p+mn_p)-(l_d+\omega_p-x(t)\omega_p+\omega_r) g(t)-n_d g(t)^2  \right)\\
				&  \cdot \frac{1}{g(t)}(1-x(t))p^1(t) +p^2(t).
	\end{split} 
\end{align}
\end{subequations}
Now we note that the equation \eqref{eq:xchar} is independent of $p^1, p^2$ and $z$, and the equations \eqref{eq:p1char}, \eqref{eq:p2char}, \eqref{eq:zchar} can be rewritten in the form
\begin{align}\label{eq:matrixp1p2z}
	\begin{pmatrix}
		\dot p^1(t) \\ \dot p^2(t)\\ \dot z(t)
	\end{pmatrix}
	= A(t)
	\begin{pmatrix}
		p^1(t) \\ p^2(t) \\ z(t)
	\end{pmatrix}+b(t),
\end{align}
where the matrix $A(t)$ and the vector $b(t)$ are defined implicitly via the previous equations \eqref{eq:p1char}, \eqref{eq:p2char} and \eqref{eq:zchar}. Both $A$ and $b$ may also depend on $x(t)$, but not on $p^1(t), p^2(t)$ or $z(t)$. There always is a unique solution of \eqref{eq:matrixp1p2z} for given initial values \cite[Theorem 3.7, 3.8, 3.10]{Teschl2000}. However, we must also find a projected characteristic, i.e., a solution of \eqref{eq:xchar}, which connects a given point $(x,t)$ in the upper half plane to a point on the boundary $\R\times \{t=0\}$, where the initial condition $h$ is defined. Therefore, we must find appropriate initial conditions
\begin{align}
	p^1(0) = p^1_0, \; p^2(0) = p^2_0,\; z(0) = z_0,\; x(0) = x_0.
\end{align}
A direct comparison with \eqref{eq:HGini} shows that the initial values for $z$ and $p^1$ are given by
\begin{align}\label{eq:zp1init}
	z_0=h(x_0),\quad p_0^1=h'(x_0).
\end{align}
Since we also want the PDE \eqref{eq:HG} to hold, we should insist that
\begin{align}\label{eq:p2init}
	p^2_0 = H(p^1_0, z_0,x_0,0).
\end{align}
In our case for each $x_0\in \R$ where the initial condition \eqref{eq:HGini} is defined there is exactly one $z_0, p^1_0, p^2_0$ such that the above equations hold. To simplify the notation~(cf.~{\cite[Section 3.2.4]{Evans2010}}), we define
\begin{align}\label{eq:charInitial}
	q(x_0) \coloneqq (x_0,h'(x_0),H(h'(x_0),h(x_0),x_0,0),h(x_0))
\end{align}
as the initial condition for the ordinary differential equations (ODEs) given by~\eqref{eq:char}. We also introduce the notation
\begin{align}
	\begin{matrix}
		x(t) = x(q(x_0),t)&z(t) = z(q(x_0),t)\\
		p^1(t) = p^1(q(x_0),t) & p^2(t) = p^2(q(x_0),t)\\
	\end{matrix}
\end{align}
to display the dependence of the solution of \eqref{eq:char} on the initial condition $q(x_0)$. Moreover, as the solution to the projected characteristic $x(t)$ does not depend on $q^2$, $q^3$ and $q^4$ but only on $q^1=x_0$, we will also write
\begin{align}
x(t) = x(x_0,t).
\end{align}
Since \eqref{eq:xchar} nonlinear in $x$, we cannot a-priori expect to cover the whole upper half plane $\R\times \R_{\ge 0}$ with projected characteristics. Indeed, if we retrace the projected characteristic which passes through $(x,t)$ by replacing $t$ with $-t$, we might get a blow-up of $x(t)$ before we reach the time $t=0$. However, as we show next, there is a sufficiently large subset of $\R\times \R_{\ge 0}$, which can be covered by characteristics.

\begin{theorem}[Global Invertibility] 
\label{thm:globalInvertibility}
Suppose that the initial condition~\eqref{eq:HGini}	is given by a generating function whose convergence radius $r$ is strictly greater than $1$. Then, there exists an open set $\cD$ which contains $[-1,1]\times \R_{> 0}$ such that for each $(\bar x,\bar t)\in \cD$ there exists a unique $x_0 \in (-r,r)$ such that
	\begin{align}\label{eq:globalInvertibility}
		\bar x = x(x_0,\bar t)
	\end{align}
	and the mapping $(\bar x,\bar t)\mapsto x_0$ is $C^k$.
\end{theorem}

\begin{proof}
We define $\cD$ by
\begin{align*}
\cD\coloneqq \{ (x,t)\in \R\times\R_{\ge 0}: (x,t) = (x(x_0,\bar t),\bar t),\ x(x_0,\cdot ) \text{ solves }\eqref{eq:xchar}, |x_0|<r,\bar t>0 \},
\end{align*}
as shown in Figure \ref{fig:CharacteristicsDandTrappingRegion}(a). Because \eqref{eq:xchar} is locally Lipschitz continuous, trajectories cannot cross each other in the extended phase space. As we can consider our domain $\cD$ as part of the extended phase space, for each $(\bar x,\bar t)\in \cD$ there exists at most one $x_0\in (-r,r)$ such that \eqref{eq:globalInvertibility} is satisfied. Furthermore, the existence of such a $x_0$ follows directly from the definition of $\cD$. Moreover, $\cD$ is open and the mapping defined in the theorem is $C^k$. It remains to show that $[-1,1]\times \R_{>0}$ is contained in $D$.\\

Consider any $(\bar x,\bar t)\in [-1,1]\times\R_{>0}$. We will show that $(\bar x,\bar t)\in \cD$ by reversing the independent variable $t$ and then retracing the characteristic to find a point where it intersects with $[-1,1]\times \{0\}$. So if we set $\tilde t = -t$ and $\tilde x=x$, \eqref{eq:xchar} becomes
\begin{align}
\label{eq:xcharRev}
\frac{\mathrm d}{d \tilde t}\tilde x(\tilde t)	&= (\tilde x(\tilde t)-1)\left( -l_d-\omega_p-\omega_r-n_d g(\tilde t)+\left(\omega_p+\frac{2l_p+mn_p}{g(\tilde t)}\right)\tilde x(\tilde t) \right)\eqqcolon f(\tilde x,\tilde t).
\end{align}
If we solve this equation with initial condition $\tilde x(-\bar t) = \bar x$, then $\tilde x(\tilde t)\in [-1,1]$ for all $\tilde t \in [-\bar t ,0]$. This is because for $\tilde x\le -1$ we have $f(\tilde x,\tilde t)\ge 0$ and for $\tilde x=1$ $f(\tilde x,\tilde t)=0$. So as trajectories are not allowed to cross, the solution $\tilde x(\tilde t)$ has to stay in $[-1,1]$, thus especially $\tilde x(0)\in [-1,1]$; see also Figure \ref{fig:CharacteristicsDandTrappingRegion}(b). So after having chosen $x_0=\tilde x(0)$ we can conclude that this is the required initial value to reach $(\bar x,\bar t)$.
\end{proof}

\begin{figure}
	\centering
	\begin{overpic}[width = 0.9\textwidth]
		{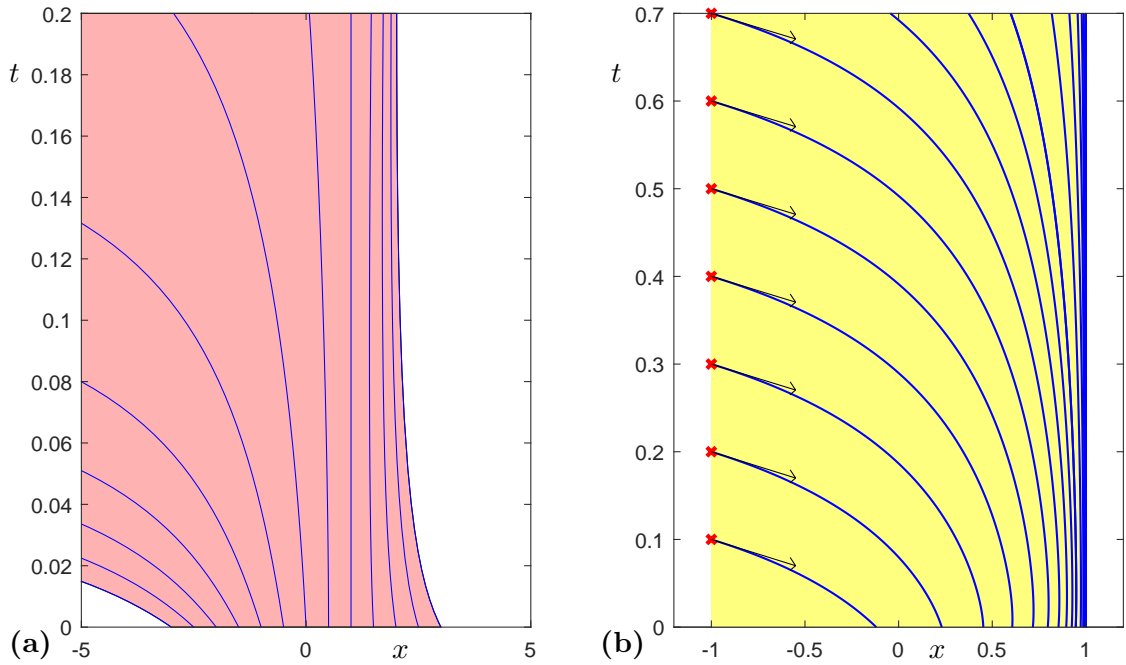}
		\put(-2,0.5){\textbf{(a)}}		
		\put(52,0.5){\textbf{(b)}}
		\put(33,0){$x$}
		\put(-2,53){$t$}
		\put(82,0){$x$}
		\put(53,53){$t$}
	\end{overpic}
	\captionsetup{aboveskip=7pt}
	\caption{\label{fig:CharacteristicsDandTrappingRegion} (a) The region $\cD$ in red. (b) The trapping region for the characteristics in yellow. If $\tilde x(\bar t)\in [-1,1]$ then $\tilde x$ stays in the yellow region upon increasing $\tilde t$. (Parameter values: $m=3, n_d=1, \omega_p = 1, l_p = 0, n_p = 1, \omega_r = 1, l_d = 1, n_r = 1, l_r=1$, Initial condition: $h(x) = \frac{2}{3} \sum\limits_{k=0}^\infty \left(\frac{x}{3}\right)^k$ ).}
\end{figure}

So we just proved that there is a certain part $\cD$ of the upper half plane such that each point $(x,t)\in \cD$ gets covered by an unique projected characteristic starting at $x_0(x,t)$. As usual when solving a PDE with the method of characteristics, we now define
\begin{align}
	\label{eq:charSol}
	G(x,t) & \coloneqq  z(q(x_0(x,t)),t)\\
	p^1(x,t) &\coloneqq p^1(q(x_0(x,t)), t)\\
	p^2(x,t) &\coloneqq p^2(q(x_0(x,t)), t)
\end{align}
for all $(x,t)\in \cD$ and all $(x,0)\in \partial \cD$. Now we get to the main result of this section.

\subsection{Existence Results}

\begin{theorem}[Global Existence Theorem for the Local Problem]
\label{thm:globexisloc}
Assume that the initial condition \eqref{eq:HGini} is given by a generating function whose convergence radius is strictly greater than $1$. Suppose additionally that $g\in C^k(\R_{\ge 0})$ with $g>0$. Then, the function $G$ is well defined by \eqref{eq:charSol}. Furthermore, it is $k$ times continuously differentiable and solves the Problem \ref{prob:H} on an open domain $\cD$ containing $[-1,1]\times \R_{>0}$.
\end{theorem}
$ $\\ \\
\begin{proof} $ $ \\
	\underline{Step 1: Well-definedness}
	Let $(\bar x,\bar t)\in \cD$. The well-definedness of $G(\bar x,\bar t)$ in \eqref{eq:charSol} follows from the uniqueness of a $x_0\in (-r,r)$ such that $\bar x=x(x_0,\bar t)$, the uniqueness of a initial condition for \eqref{eq:char} for given $x_0$ and the fact that the characteristic ODEs \eqref{eq:char} are uniquely solvable.\\
	
	\underline{Step 2: Differentiability}
	Remember that for a given pair $(x,t)\in D$ we defined $G(x,t)$ in the following way:
	\begin{align}\label{eq:intermediateGdef}
		(x,t)\overset{(i)}{\longmapsto} (x_0(x,t),t)\overset{(ii)}{\longmapsto} (q(x_0(x,t)),t)\overset{(iii)}{\longmapsto} z(q(x_0(x,t)),t)\eqqcolon G(x,t)
	\end{align}
	As Theorem \ref{thm:globalInvertibility} shows, the mapping $(i)$ is $C^k$. Furthermore, we recall that the initial condition $q(x_0)$ of the ODE \eqref{eq:char} is given by $q(x_0)=(x_0,h'(x_0),H(h'(x_0), h(x_0), x_0,0), h(x_0))$. Using that $h$ is a power series and the definition of $H$ it can easily be checked that the mapping $(ii)$ is $k$-times continuously differentiable. Finally, it is a well known fact, that $(iii)$ is of class $C^k$; see for example \cite[Theorem 9.7]{Amann1990}. As a result, $G$, which is the composition of $(i), (ii)$ and $(iii)$, is itself an element of $C^k(\cD)$.\\
	
	\underline{Step 3: Boundary conditions}
	We have that $G(x,0) = z(q(x_0(x,0)),0) = z(q(x),0) = h(x)$ by \eqref{eq:zp1init}. So $G$ satisfies the boundary conditions \eqref{eq:HGini}.\\
	
	\underline{Step 4: Solution inside the domain}
	This part of the proof is similar to standard proofs that a function $G$, which is defined as in \eqref{eq:charSol}, locally solves the PDE \eqref{eq:HG}. As in our case characteristics do not cross and each point in the domain $D$ can be reached by a characteristic, these proofs extend to global solvability. One of these proofs can be found in \cite[Section 3.2, Theorem 2]{Evans2010}.
\end{proof}

\begin{remark}
One could attempt to prove a slightly stronger version of Theorem \ref{thm:globexisloc}. In fact, it can be shown that if $g\in C^\omega(\R_{\ge 0})$ then also $G\in C^\omega(\cD)$. However, this does not show that the property of $G(x,t)$ being a power series for $t=0$ with respect to $x$ also holds for later times. To show this, a more profound analysis of the intermediate steps in the definition of $G$ in \eqref{eq:intermediateGdef} is required. In particular, one has to show that
\begin{enumerate}
	\item \label{item:step1} for every $\bar t>0$ the map $(x,\bar t)\mapsto x_0(x,\bar t)$ is a power series in $x\in [-1,1]$ around $x=0$, which as we already know only takes values in the interval $[-1,1]$,
	\item \label{item:step2} the map $x_0\mapsto q(x_0)$ is a power series for $x_0\in [-1,1]$,
	\item \label{item:step3} the map $(\xi,\bar t)\mapsto z(\xi,\bar t)$ is a power series in $\xi$.
\end{enumerate}
As we know that $x(t)\equiv 1$ is a solution to the characteristic equation \eqref{eq:xchar}, we can explicitly calculate other solutions which leads us to the fact that the first step in \eqref{eq:intermediateGdef} is a power series. It also directly follows from the definition of $q$ that the second step is also given by a power series. The difficulty, however, lies in showing that the map in step \ref{item:step3} is a power series.
\end{remark}

To sum up, under the assumptions
\begin{itemize}
	\item $g\in C^k(\R_{\ge 0})$, $g(t)>0$,
	\item The initial condition is given by a generating function whose convergence radius $r$ is strictly greater than $1$,
\end{itemize}
we have proved that there exists an open set $\cD$, which contains $[-1,1]\times \R_{>0}$, and a function $G\in C^k(\cD)$ which is a solution of Problem \ref{prob:H}. Because all the projected characteristics lie completely in $\cD$ and do not intersect and because the ODE \eqref{eq:xchar} is uniquely solvable, we obtain that $G$ is even the unique $C^2(\cD)$ solution \cite[Section 3.2, Theorem 1]{Evans2010}.

\begin{theorem}[Global Existence Theorem for the Nonlocal Problem]\label{thm:globexisnonloc}
	Suppose that the convergence radius of $\tilde h$ in \eqref{eq:InitialCond} is strictly greater than $1$ and $\tilde h'(1)\neq 0$. Then, there exists an open set $\cD$, which contains $[-1,1]\times \R_{>0}$, and a function $G\in C^\infty(\cD)$ which is a solution of the nonlocal Problem \ref{prob:main}.
\end{theorem}

\begin{proof}
Let $g\colon \R_{\ge 0}\to \R$ be a solution of the initial value problem~\eqref{eq:diffeqg}, so $g$ satisfies
\begin{subequations}
\label{eq:dynamg}	
\begin{align}
		\label{eq:dynamgdiff}
		g'(t) &= -n_d (g(t))^2-bg(t)+c,\\
		g(0) &= \tilde h'(1) > 0,
\end{align}
\end{subequations}
with $b=l_d+n_p+n_r\ge 0$, $c=2(l_p+l_r+m(n_p+n_r))\ge 0$. It is easy to verify that the solution $g(t)$ to this initial value problem will be positive - a fact which is independent of $n_d,b,c$. Furthermore, because the right-hand side of \eqref{eq:dynamgdiff} is $C^\infty$ with respect to $g$, we conclude that $g\in C^\infty(\R_{\ge 0})$, too. Thus, the assumptions of Theorem \ref{thm:globexisloc} are satisfied, which means that there exists a solution $G\in C^\infty(\cD)$ of Problem \ref{prob:H}. In view of the Problem Equivalence \ref{lem:probequiv}, $G$ is also a solution of Problem \ref{prob:main}.
\end{proof}

By Lemma \ref{lem:probequiv}, uniqueness of the $C^2(\cD)$ solution for the nonlocal problem also follows directly from the uniqueness of solutions of the local problem.\medskip 

We want to close this section by giving a numerical visualization of a solution of our Problem \ref{prob:main}. Figure \ref{fig:GPlotAndChar}(a) shows the three dimensional plot of the solution for given parameter values and given initial condition. Although we have only calculated the function $G$ up to $t=0.2$, we can nevertheless conjecture the existence of a steady state. Figure \ref{fig:GPlotAndChar}(c) shows a contour plot of the same function $G$. The existence of steady states will be discussed in detail in Section~\ref{sec:steadystates} followed by some more numerical simulations. Finally, Figure \ref{fig:GPlotAndChar}(b) shows some projected characteristics, along which we computed the value of $G$.

\begin{figure}[h]
	\centering
	\begin{overpic}[width = 0.9\textwidth]
		{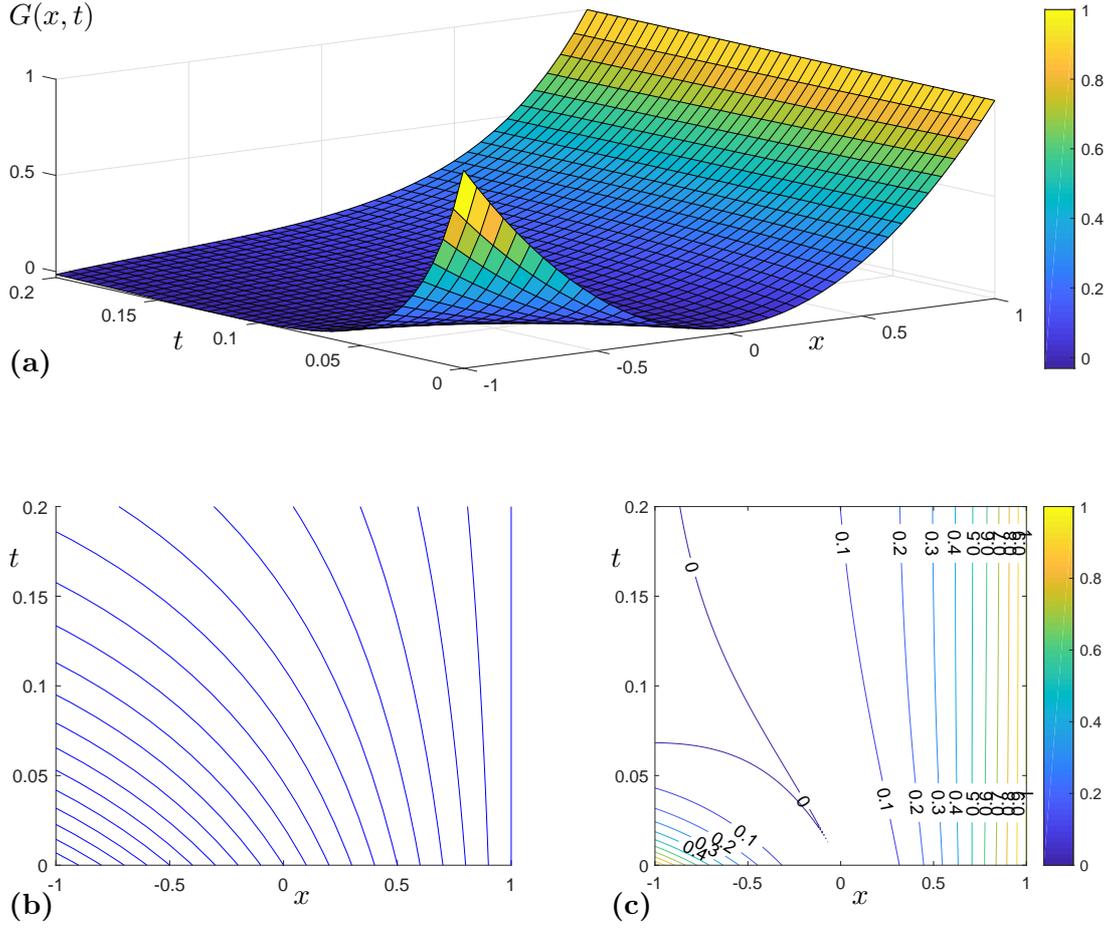}
		\put(0,48){\textbf{(a)}}		
		\put(0,-2){\textbf{(b)}}
		\put(55,-2){\textbf{(c)}}
		\put(73,50){$x$}
		\put(15,50){$t$}
		\put(0,80){$G(x,t)$}
		\put(26,-1){$x$}
		\put(0,30){$t$}
		\put(77,-1){$x$}		
		\put(55,30){$t$}
	\end{overpic}
	\captionsetup{aboveskip=15pt}
	\caption{\label{fig:GPlotAndChar} (a) Three dimensional plot of $G(x,t)$. (b) Projected characteristics along which we computed the values of $G(x,t)$. (c) Contour plot of the same function $G$. (Parameter values: $m=3,\ n_d=1,\ \omega_p=1,\ l_p=0,\ n_p=1,\ \omega_r=1,\ l_d=1,\ n_r=1,\ l_r=1$, initial condition: $h(x) = x^2$)}
\end{figure}

\newpage

\section{Steady States}\label{sec:steadystates}

In this section we investigate the existence of steady states of the PDE \eqref{eq:MainEquation}. We will restrict ourselves to the spatial domain $x\in [-1,1]$.

\subsection{Problem and General Strategy}

In our case a steady state $G^*=G^*(x)$ of the PDE \eqref{eq:MainEquation} has to 
satisfy the equation
\begin{align}\label{eq:steadymain}
	\begin{split}
		0 &= (x-1)\left[x\left( \omega_p+\frac{2l_p+n_p m}{G^\star_x(1)}\right)-\omega_r-\omega_p-l_d-n_dG^\star_x(1)\right]G^\star_x(x)\\ &+\left[ (x-1)\left(\omega_rG^\star_x(1)+2l_r+n_rm\right)-n_r-n_p\right] G^\star(x) +(n_r+n_p)x^m.
	\end{split}
\end{align}
As we already know from Proposition \ref{prop:DiffGlAnf}, $g(t)=G_x(1,t)$ satisfies the initial value problem
\begin{align}\label{eq:diffglanf2}
	g'(t)&=-n_d(g(t))^2-bg(t)+c\\ g(0)&=\tilde h'(1)> 0,
\end{align}
with $b=l_d+n_p+n_r\ge 0$ and $c=2(l_p+l_r+m(n_p+n_r))\ge 0$.
Assuming that $G(x,t)$ is a generating function for all $t\ge 0$ and $\lim_{t\to \infty} G_x(1,t) =0$, we can directly deduce that $G(x,t)$ converges to $G^\star(x)\equiv 1$. Similarly, if $\lim_{t\to \infty} G_x(1,t)=\infty$, which may happen if $n_d=b=0$ and $c>0$, there is obviously no steady state. For all the remaining cases, \eqref{eq:diffglanf2} yields a unique positive equilibrium point for $G_x(1,t)$. Substituting this equilibrium point for $G_x(1,t)$ into \eqref{eq:steadymain}, we can formulate the mathematical problem of finding a steady state as follows:

\begin{problem}\label{prob:steadystate}
	Find a continuous function $G^\star \colon [-1,1]\to \R$ such that
	\begin{align}\label{eq:steadyimplicit}
		0=(x-1)(x c_1-c_2)G^\star_x(x)+((x-1)c_3-c_4)G^\star (x)+c_4x^m,
	\end{align}
	where the constants $c_1,\dots,c_4$ are given by
	\begin{align*}
		c_1 &= \omega_p+\lim_{t\to \infty} \frac{2l_p+n_pm}{G_x(1,t)},&c_2 &= \omega_r+\omega_p+l_d+ \lim_{t\to \infty} n_d G_x(1,t),\\
		c_3 &= \lim_{t\to \infty}\omega_rG_x(1,t)+2l_r+n_rm, &c_4 &=	n_r+n_p.
	\end{align*}
\end{problem}

\begin{remark}
Given a solution $G^\star$ of Problem \ref{prob:steadystate}, which satisfies $G^\star_x(1)=0$, one always needs to be careful, since it does not automatically need to be a steady state of the PDE \eqref{eq:MainEquation}. That is because the right-hand side of the PDE may not be well defined. Correspondence between solutions of Problem \ref{prob:steadystate} and steady states of PDE \eqref{eq:MainEquation} is only guaranteed, if $G^\star_x(1)\neq 0$.
\end{remark}

\begin{remark}
A solution $G^\star$ of Problem \ref{prob:steadystate} may satisfy $G^\star_x(1)=0$ even though we have already excluded $\lim_{t\to \infty} G_x(1,t) = 0$. Consider for example the case $c_3 = c_4 = 0$.
\end{remark}

The equation \eqref{eq:steadyimplicit} is an implicit differential equation with singular points at $x=1$ and $x=\frac{c_2}{c_1}$. Under the assumption $c_4\neq 0$, inserting the singular points into \eqref{eq:steadyimplicit} yields $G^\star(1)=1$ and 
\begin{align}
	G^\star\left(\frac{c_2}{c_1}\right) = \frac{-c_4 (\frac{c_2}{c_1})^m}{(\frac{c_2}{c_1}-1)c_3-c_4}.
\end{align}
At these points, however, the value of $G^\star_x$ cannot be determined by $x$ and $G^\star(x)$. At the nonsingular points we can divide by the pre-factor of $G^\star_x$, provided that either $c_1>0$ or $c_2>0$, to get
\begin{align}\label{eq:steadyexplicit}
	G^\star_x(x) = \underbrace{\frac{-((x-1)c_3-c_4)}{(x-1)(x c_1-c_2)}}_{\eqqcolon f(x)} G^\star(x) + \underbrace{\frac{-c_4 x^m}{(x-1)(x c_1 -c_2)}}_{\eqqcolon b(x)},
\end{align}
which is a linear explicit differential equation. To find a solution to Problem \ref{prob:steadystate} we will take the ansatz of variation of constants at the nonsingular points and then try to continue the solution up to the singular points. The ansatz of variation of constants is given by the formula
\begin{align}\label{eq:varofconstgen}
	G^\star(x) = \left( G(x_0) + \int_{x_0}^x e^{-a(s)} b(s)\ \mathrm ds\right) e^{a(s)},
\end{align}
where
\begin{align}
	a(x) = \int_{x_0}^x f(s) \ \mathrm ds.
\end{align}

Unfortunately, to address Problem \ref{prob:steadystate} in general, we have to distinguish between many cases. Since most of those cases are very similar, we will only present one case in full detail. Specifically, we will consider the most interesting case when we have two different singular points in our domain $[-1,1]$, that means if $\frac{c_2}{c_1}<1$. Furthermore, we will also assume that $c_4>0$.

\subsection{Deriving a formula for a steady state}

In this section we cover the case $0\le c_2<c_1$, $c_4>0$. As shown in Figure \ref{fig:SteadyStateTrajectories}, there are two singularities in our domain $[-1,1]$; one at $x=\frac{c_2}{c_1}$ and the other one at $x=1$. First of all, we will look for solutions $G_1^\star$ of the ODE \eqref{eq:steadyexplicit} in the interval $[-1,\frac{c_2}{c_1})$. Because there is no initial value given we expect to get a one-parameter family of solutions, which we will denote by $G_{1,p_1}^\star$. Similarly, let $G_{2,p_2}^\star$ be the family of solutions in the interval $(\frac{c_2}{c_1},1)$. In view of our preliminary calculations, we will then define
\begin{align}\label{eq:c1gc2G}
	G^\star(x)\coloneqq \begin{cases} G_{1,p_1}^\star(x) & \text{if }x\in [-1,\frac{c_2}{c_1})\\
	\frac{-c_4 (\frac{c_2}{c_1})^m}{(\frac{c_2}{c_1}-1)c_3-c_4} &\text{if }x=\frac{c_2}{c_1}\\
	G_{2,p_2}^\star(x) & \text{if } x\in (\frac{c_2}{c_1},1)\\
	1& \text{if } x=1\\
	\end{cases},
\end{align}
where we choose the parameters $p_1,p_2$ such that $G^\star$ is continuous on $[-1,1]$. It first seems impossible to find appropriate parameters because there are three conditions, specifically the left-sided limit of $G_1^\star$ at $x=\frac{c_2}{c_1}$, the right-sided limit of $G_2^\star$ at $x=\frac{c_2}{c_1}$ and the left-sided limit of $G_2^\star$ at $x=1$, but only two parameters. We will show that we can nevertheless find $p_1,p_2$ such that all these conditions are satisfied.\medskip

\begin{figure}[htbp]
	\centering
	\begin{overpic}[width = 0.5\textwidth]
		{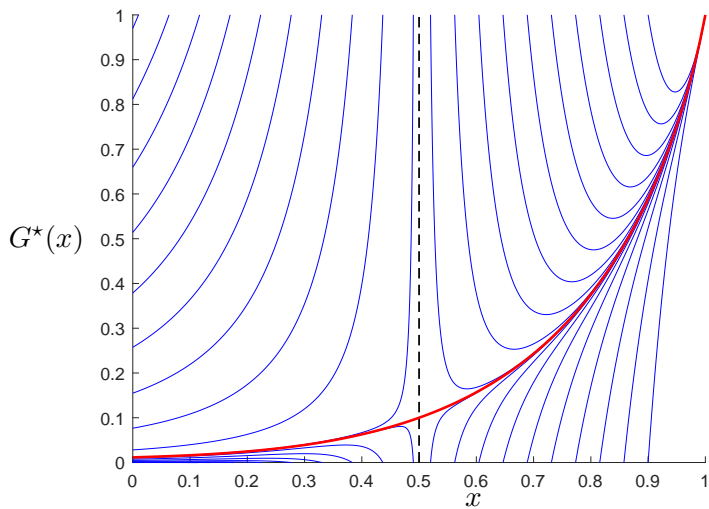}
		\put(-15,40){$G^\star (x)$}
		\put(60,-3){$x$}
	\end{overpic}
	\caption{\label{fig:SteadyStateTrajectories} Trajectories of \eqref{eq:steadyexplicit} in the extended phase in case $0\le c_2<c_1,c_4>0$. Coefficients: $c_1=2,c_2=1,c_3=1,c_4=2,m=3$. Singularities at $x=\frac{c_2}{c_1}=\frac{1}{2}$ and $x=1$. We are looking for the red trajectory, which is the only one that is continuous on the whole interval.}
\end{figure}

First of all, as we want to use the approach of variation of constants and thus need to integrate $f$, it is helpful to rewrite it as
\begin{align}
	f(x) &= \frac{-((x-1)c_3-c_4)}{(x-1)(x c_1-c_2)} = \frac{\alpha}{x-1} + \frac{\beta}{x-\frac{c_2}{c_1}},
\end{align}
with constants
\begin{align}
	\alpha &=\frac{c_4}{c_1-c_2}>0, \quad \beta = \frac{-c_1c_3+c_2c_3-c_1c_4}{(c_1-c_2)c_1}<0.
\end{align}
To keep notation simple, we first define $w_1(x) \coloneqq \exp(a(x))$. A short calculation based on the choice $x_0=1$ and the fact $x<\frac{c_2}{c_1}$ shows that
\begin{align}
	w_1(x) =(1-x)^\alpha \left(\frac{c_2}{c_1}-x \right)^\beta 2^{-\alpha}\left(\frac{c_2}{c_1}+1\right)^{-\beta}.
\end{align}
Substituting this into the ansatz of variation of constants \eqref{eq:varofconstgen} yields
\begin{align}\label{eq:c1gc2G1}
	G_1^\star(x) = \left( G^\star_1(-1)+\int_{-1}^x \frac{b(s)}{w_1(s)}\ \mathrm ds\right) w_1(x).
\end{align}
For every $G_1^\star(-1)$ this is a solution of \eqref{eq:steadyexplicit} and as a result also of \eqref{eq:steadyimplicit} for $x\in [-1,\frac{c_2}{c_1})$. However, as $x\uparrow \frac{c_2}{c_1}$ we have that $|w_1(x)|\to \infty$. As a consequence, the only chance to continuously extend the solution to $x=\frac{c_2}{c_1}$ is if 
\begin{align}
	G_1^\star(-1)+\int_{-1}^x \frac{b(s)}{w_1(s)}\ \mathrm ds \to 0 \quad \text{as } x\uparrow \frac{c_2}{c_1}.
\end{align}
Thus, the obvious choice for $G_1^\star(-1)$ is given by
\begin{align}\label{eq:c1gc2G-1}
	G_1^\star(-1) = -\int_{-1}^\frac{c_2}{c_1} \frac{b(s)}{w_1(s)}\ \mathrm ds.
\end{align}
To check that the integral is actually converging we write
\begin{align}
	G_1^\star(-1) &= -\lim_{x\uparrow \frac{c_2}{c_1}}\int_{-1}^x (1-s)^{-\alpha}\left(\frac{c_2}{c_1}-s\right)^{-\beta}2^{\alpha} \left(\frac{c_2}{c_1}+1\right)^{\beta} \frac{-c_4 s^m}{(s-1)(c_1 s-c_2)}\ \mathrm ds\\
	\label{eq:G1star-1int}
	&= -\lim_{x\uparrow \frac{c_2}{c_1}} \int_{-1}^x \left(\frac{c_2}{c_1}-s\right)^{-\beta-1} u(s)\ \mathrm ds
\end{align}
with
\begin{align}
	u(x) = (1-x)^{-\alpha} 2^\alpha \left(\frac{c_2}{c_1}+1\right)^{\beta} \frac{c_4 x^m}{c_1(x-1)}.
\end{align}
As $u$ is a continuous function on the closed interval $[-1, \frac{c_2}{c_1}]$ we can find a uniform upper bound to $|u(x)|$ for $x\in [-1,\frac{c_2}{c_1}]$. Using that $u(\frac{c_2}{c_1})\neq 0$ an estimation of the integral in \eqref{eq:G1star-1int} shows that the limit exists if and only if $\beta <0$. But as this condition is always true, $G_1^\star(-1)$ is well defined and thus our parameter $p_1$ is set. This shows that $G_1^\star$ defined in \eqref{eq:c1gc2G1} with $G_1^\star(-1)$ given by \eqref{eq:c1gc2G-1} is a solution of \eqref{eq:steadyimplicit} for $x\in [-1,\frac{c_2}{c_1})$.\\

Now we dedicate ourselves to the second segment of $G^\star$, which is the interval $(\frac{c_2}{c_1},1)$. For this part we choose any $y\in (\frac{c_2}{c_1},1)$ and use it as the point where the initial value will be set. Again we define $w_2(x)\coloneqq \exp(a(x))$ with $x_0=y$. Under the assumption that $x\in (\frac{c_2}{c_1},1)$, a short calculation yields that
\begin{align}
	w_2(x) = (1-x)^\alpha\left(x-\frac{c_2}{c_1}\right)^\beta (1-y)^{-\alpha}\left(y-\frac{c_2}{c_1}\right)^{-\beta},
\end{align}
so after substitution of this into the variation-of-constants formula \eqref{eq:varofconstgen}, it reads as
\begin{align}\label{eq:c1gc2G2}
	G^\star_2(x) = \left( G^\star_2(y)+\int_y^x \frac{b(s)}{w_2(s)}\ \mathrm ds \right) w_2(x).
\end{align}
As in the above case, if $x\downarrow \frac{c_2}{c_1}$ we have that $|w_2(x)|\to \infty$, consequently we should obviously choose
\begin{align}
	G^\star_2(y) = -\int_y^\frac{c_2}{c_1} \frac{b(s)}{w_2(s)}\ \mathrm ds.
\end{align}
Again, a short computation yields that the integral is converging and thus that $G^\star_2(y)$ is well defined. This condition fixes our second parameter $p_2$, whereby we can now define $G^\star$ as anticipated in \eqref{eq:c1gc2G}.
\begin{remark}
	The definition of $G_2^\star$ does not depend on the choice of $y\in (\frac{c_2}{c_1},1)$.
\end{remark}

So far we have only investigated necessary conditions for $G^\star$ and have found a unique $G^\star$, which satisfies these conditions. It remains to show that this $G^\star$ actually is a solution of our Problem \ref{prob:steadystate}, which means that we have to show that $G^\star$ is continuous, particularly at $x=\frac{c_2}{c_1}$ and $x=1$, as continuity at the other points is trivial. In total, the following three conditions need to be satisfied for $G^\star$ being a solution to Problem \ref{prob:steadystate}:

\begin{enumerate}[I]
	\item \label{item:steadycond1}$G^\star(\frac{c_2}{c_1}) = \lim\limits_{x\uparrow \frac{c_2}{c_1}} G_1^\star(x)\eqqcolon G^\star_-(\frac{c_2}{c_1})$,
	\item \label{item:steadycond2}$G^\star(\frac{c_2}{c_1}) = \lim\limits_{x\downarrow \frac{c_2}{c_1}} G_2^\star(x)\eqqcolon G^\star_+(\frac{c_2}{c_1})$,
	\item \label{item:steadycond3}$G^\star(1) = \lim\limits_{x\uparrow 1} G_2^\star(x) \eqqcolon G^\star_-(1)$.
\end{enumerate}

A short calculation utilizing L'Hospital's rule verifies \ref{item:steadycond1}:

\begin{align*}
	G^\star_-\left(\frac{c_2}{c_1}\right) &= \lim_{x\uparrow \frac{c_2}{c_1}} \left( -\int_{-1}^\frac{c_2}{c_1} \frac{b(s)}{w_1(s)}\ \mathrm ds+\int_{-1}^x \frac{b(s)}{w_1(s)}\ \mathrm ds\right) w_1(x)\\
	&= \lim_{x\uparrow \frac{c_2}{c_1}} b(x) \frac{\omega_1^{-1}(x)}{\frac{\mathrm d}{\mathrm dx}\omega_1^{-1}(x)}\\	
	&= \lim_{x\uparrow \frac{c_2}{c_1}} \frac{c_4x^m}{c_1(1-x)(\frac{c_2}{c_1}-x)} \left(\frac{(1-x)^{-\alpha}(\frac{c_2}{c_1}-x)^{-\beta}}{\alpha (1-x)^{-\alpha-1}(\frac{c_2}{c_1}-x)^{-\beta}+\beta (1-x)^{-\alpha}(\frac{c_2}{c_1}-x)^{-\beta-1}}\right)\\
	&=\lim_{x\uparrow \frac{c_2}{c_1}} \frac{c_4 x^m}{c_1 \alpha(\frac{c_2}{c_1}-x)+c_1\beta (1-x)}\\
	&=\frac{-c_4(\frac{c_2}{c_1})^m(c_1-c_2)}{(-c_1c_3+c_2c_3-c_1c_4)(\frac{c_2}{c_1}-1)} = \frac{-c_4 (\frac{c_2}{c_1})^m}{(\frac{c_2}{c_1}-1)c_3-c_4} = G^\star\left(\frac{c_2}{c_1}\right)
\end{align*}

A similar computation shows \ref{item:steadycond2}. To show \ref{item:steadycond3}, we write
\begin{align}
	&G^\star_2(y)+\int_y^x \frac{b(s)}{w_2(s)}\ \mathrm ds\\
	&=G^\star_2(y)+\int_y^x (1-s)^{-\alpha-1}\underbrace{\left(s-\frac{c_2}{c_1}\right)^{-\beta}(1-y)^{\alpha}\left(y-\frac{c_2}{c_1}\right)^\beta\ \frac{c_4 s^m}{(sc_1-c_2)}}_{\eqqcolon u(s)} \mathrm ds.
\end{align}
Note that $u$ is continuous at $x=1$. Our first claim is that
\begin{align}
	\frac{G^\star_2(y)+\int_y^x \frac{b(s)}{w_2(s)}\ \mathrm ds}{\frac{1}{\alpha}(1-x)^{-\alpha}} \to u(1)\quad \text{as } x\uparrow 1.
\end{align}
To see this we use L'Hospital's rule. Differentiating the numerator and the denominator yields
\begin{align}\label{eq:useLHospital}
	\lim_{x\uparrow 1} \frac{G^\star_2(y)+\int_y^x \frac{b(s)}{w_2(s)}\ \mathrm ds}{\frac{1}{\alpha}(1-x)^{-\alpha}} = \lim_{x\uparrow 1} \frac{(1-x)^{-\alpha-1}u(x)}{(1-x)^{-\alpha-1}} = u(1).
\end{align}

Now we can use this claim to prove \ref{item:steadycond3}. We have
\begin{align}
	G^\star_-(1) &= \lim_{x\uparrow 1} G_2^\star(x)\\
	&=\lim_{x\uparrow 1} \left( G^\star_2(y)+\int_y^x \frac{b(s)}{w_2(s)}\ \mathrm ds \right) w_2(x)\\
	&=\lim_{x\uparrow 1} \left( G^\star_2(y)+\int_y^x \frac{b(s)}{w_2(s)}\ \mathrm ds \right) (1-x)^\alpha\left(x-\frac{c_2}{c_1}\right)^\beta (1-y)^{-\alpha}\left(y-\frac{c_2}{c_1}\right)^{-\beta}\\
	&=\frac{1}{\alpha} u(1)\lim_{x\uparrow 1} \left(x-\frac{c_2}{c_1}\right)^\beta (1-y)^{-\alpha}\left(y-\frac{c_2}{c_1}\right)^{-\beta}\\
	&=\frac{1}{\alpha}\frac{c_4}{c_1-c_2} = 1.
\end{align}
This proves \ref{item:steadycond3} and thus concludes the proof that in the considered case, $G^\star$ as constructed above, is the unique solution to Problem \ref{prob:steadystate}.\medskip

We remark that the choice of $G^\star_2(y)$ does not have an influence on $\lim_{x\uparrow 1} G_2^\star(x)$, as it vanishes in the calculation of the limit in \eqref{eq:useLHospital}. This is also noticeable since all solutions in Figure \ref{fig:SteadyStateTrajectories} pass through $(1,1)$. Furthermore, by similar computations it can also be shown that $G^\star$ as it is defined in \eqref{eq:c1gc2G} is even continuously differentiable.

\subsection{A Summary of all Cases}\label{sec:summaryofcases}

In this section we provide an overview of the existence and properties of solutions to Problem \ref{prob:steadystate} in all cases. Firstly, note that in the case $c_1=c_2=c_3=c_4=0$ every continuous function solves the problem. However, as this case requires almost all processes to be absent, it is not of further interest. Secondly, the case $c_3=c_4=0$ admits only constant functions as solutions. In all the other cases, the existence of a solution to Problem \ref{prob:steadystate} will be collected in the following table. We will also indicate whether it satisfies the condition $G^\star(1)=1$.

\begin{table}[h]
\centering

\begin{tabular}{|c|p{5cm}|p{5cm}|}
	\hline
	Case 				& $c_4=0,\ c_3>0$ 					& $c_4>0$\\ \hline
	$0=c_1=c_2$ 		& Only $G^\star(x)\equiv 0$  & \vspace{-2mm} \begin{tabular}{p{2.2cm}|p{2cm}} $c_3=0$ & $c_3>0$\\ \hline Only $G^\star(x)\equiv 1$ if $m=0$& Unique solution with $G^\star(1) = 1$ \\   \end{tabular} \vspace{1mm} \\ \hline
	$0=c_1<c_2$			& A one-parameter family & Unique solution\\ \cline{1-1}
	$0<c_1<c_2$			& of solutions. A unique & with $G^\star(1)=1$\\ 
						& one with $G^\star(1)=1$.& \\ \cline{0-1}
	$0<c_1=c_2$			& Only $G^\star(x)\equiv 0$ &  \\ \cline{1-1}
	$0=c_2<c_1$			& & \\ \cline{1-1}
	$0<c_2<c_1$			& & \\ \hline
\end{tabular}
\caption{\label{tab:overview}An overview of the existence of solutions in all cases.}
\end{table}

\section{Numerical Simulations}

In this section we carry out numerical simulations to see, whether the solution converges to the steady state, which we calculated in the previous section. It is natural to start with a case, where most of the coefficients appearing in Problem \ref{prob:steadystate} are different from zero. In particular, Figure \ref{fig:Simulation1}(a) shows the behavior of the solution $G(x,t)$ of our initial value problem with the initial condition $h(x)=\sum_{k=0}^\infty \frac 23 (\frac{x}{3}) ^k = \frac{2}{3-x}$ in the case $0<c_1<c_2$ and $c_3,c_4>0$. Moreover, the parameters $m,n_d,\dots,l_r$ are chosen such that the solution $G^\star$ of Problem \ref{prob:steadystate} satisfies $G^\star_x(1)\neq 0$ and therefore it is a steady state. Figure \ref{fig:Simulation1}(b) shows that $G(\cdot ,t)$ converges to the steady state. As seen in Figure \ref{fig:Simulation1}(c), the difference between $G^\star$ and the solution $G(x,t)$ measured with both $\norm\cdot_\infty$ and $\norm\cdot_{L^2}$ decreases exponentially fast.\medskip

\begin{figure}[h]
	\centering
	\begin{overpic}[width=0.9\textwidth]{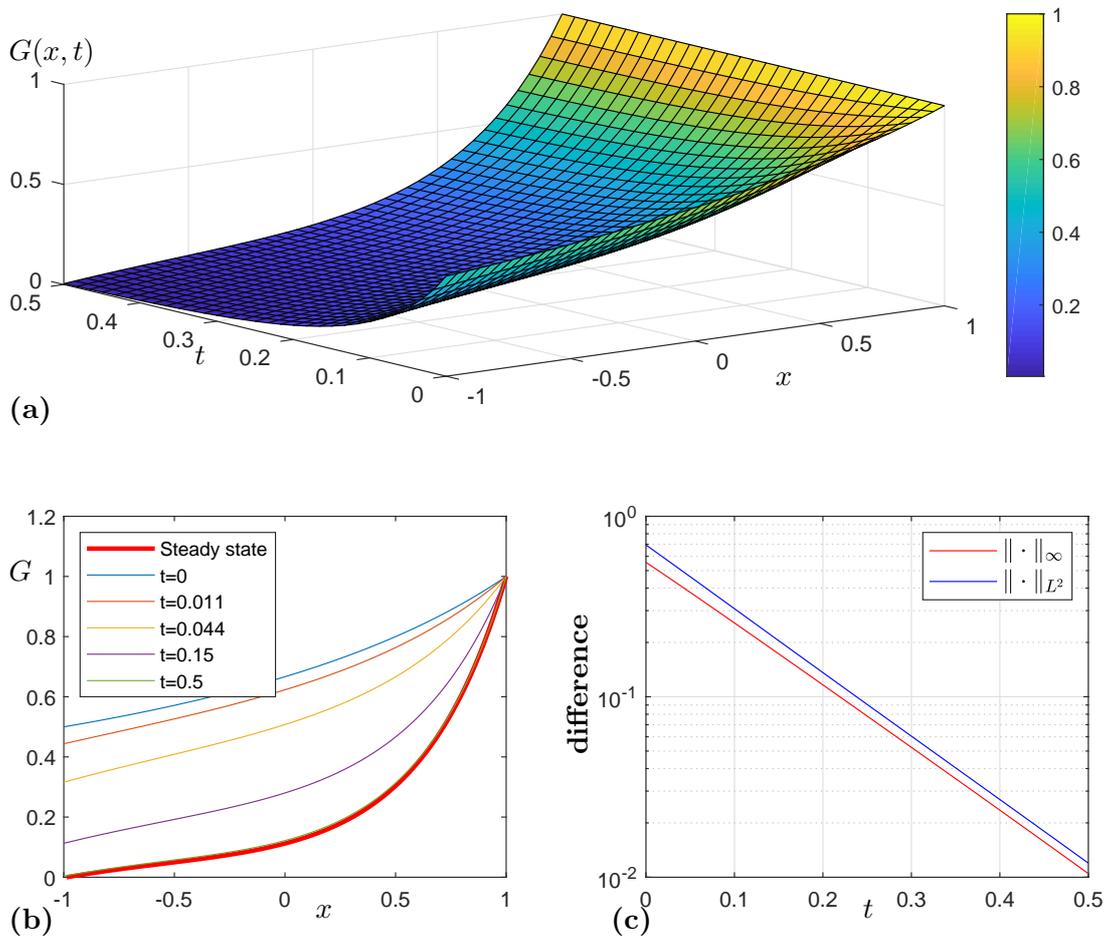}
		\put(0,45){\textbf{(a)}}
		\put(0,-2){\textbf{(b)}}
		\put(55,-2){\textbf{(c)}}
		\put(17,50){$t$}
		\put(70,48){$x$}
		\put(0,78){$G(x,t)$}
		\put(28,-1){$x$}
		\put(0,30){$G$}
		\put(78,-1){$t$}
		\put(51,14){\rotatebox{90}{\textbf{difference}}}
	\end{overpic}
	\captionsetup{aboveskip=20pt}
	\caption{\label{fig:Simulation1} (a) $3$-dimensional plot of the solution $G(\cdot,\cdot)$. (b) Solution $G(\cdot,t)$ for different values of $t$ and the steady state. (c) Difference of the solution $G(\cdot,t)$ to the steady state with respect to $\norm\cdot_\infty$ and $\norm\cdot_{L^2}$ in dependence of time. (Parameter values: $m=3, n_d=1, \omega_p=1, l_p=0, n_p=1, \omega_r=1,l_d=1,n_r=1,l_r=1$, initial condition: $h(x)=\frac 23\sum_{k=0}^\infty (\frac{x}{3})^k = \frac{2}{3-x}$)}
\end{figure}

However, we cannot always expect $G(\cdot,t)$ to converge to $G^\star$ with an exponential rate. To illustrate this, we now consider the case $c_3=c_4=0$ and $\lim_{t\to \infty} G_x(1,t)\in (0,\infty)$. The only solution $G^\star$ to Problem \ref{prob:steadystate} which also satisfies $G^\star(1)=1$ is given by $G^\star(x)\equiv 1$. A numerical simulation, as seen in Figure \ref{fig:Simulation3algebraicConvergenceTo1}(a), shows that $G(x,t)$ converges to $G^\star$. Figure \ref{fig:Simulation3algebraicConvergenceTo1}(b) displays the difference of $G^\star$ to the solution $G(x,t)$. In both the $L^\infty$ and the $L^2$ norm, this difference does not decrease exponentially fast but we can only observe convergence of algebraic order.

\begin{figure}[h!]
	\centering
	\begin{overpic}[width=0.9\textwidth]{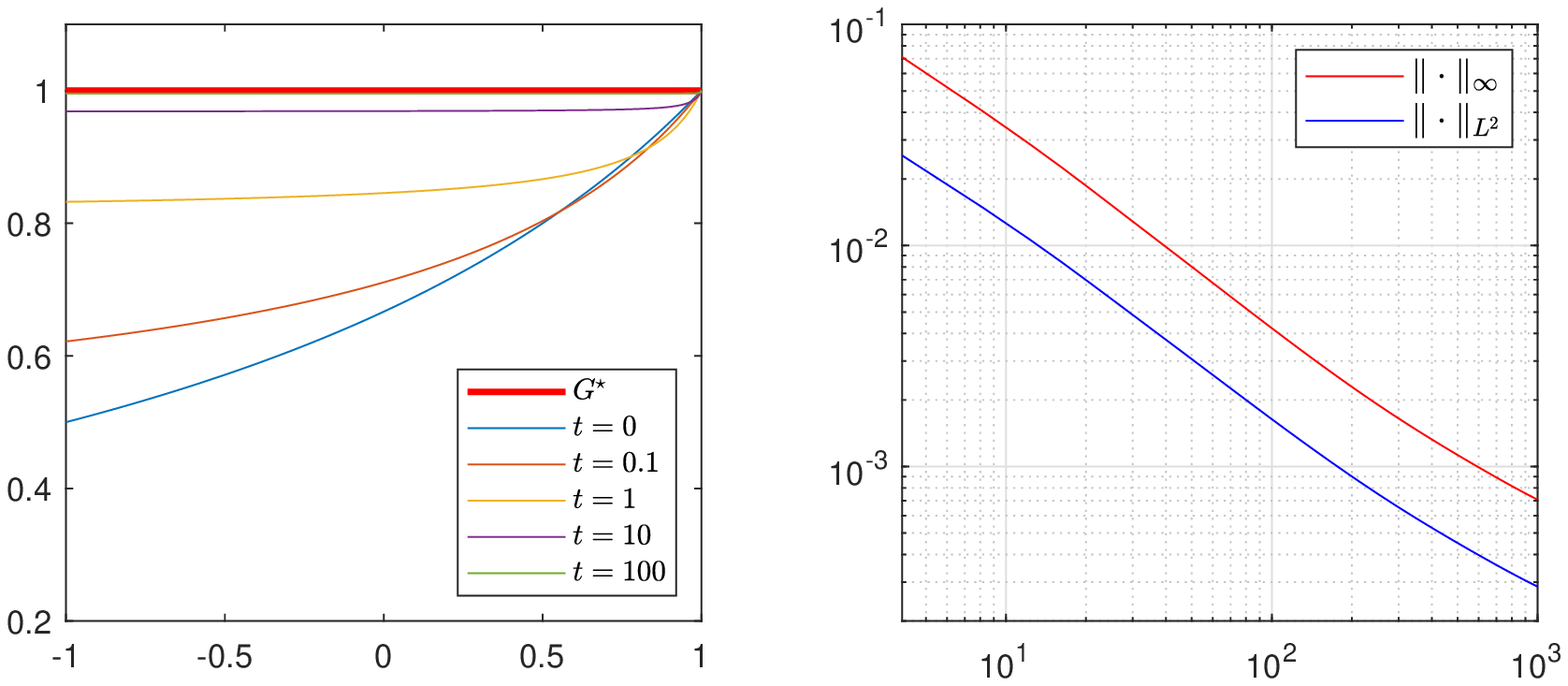}
		\put(0,-2){\textbf{(a)}}
		\put(55,-2){\textbf{(b)}}
		\put(28,-0.5){$x$}
		\put(0,33){$G$}
		\put(88,-0.5){$t$}
		\put(50.5,18){\rotatebox{90}{\textbf{difference}}}
	\end{overpic}
	\captionsetup{aboveskip=18pt}
	\caption{\label{fig:Simulation3algebraicConvergenceTo1} (a) Solution $G(\cdot ,t)$ for different values of $t$ and the solution of Problem \ref{prob:steadystate}. (b) Difference of the solution $G(\cdot,t)$ to the solution of Problem \ref{prob:steadystate} with respect to $\norm\cdot_\infty$ and $\norm\cdot_{L^2}$ in dependence of time. (Parameter values: $m=3, n_d=1, \omega_p=1, l_p=1, n_p=0, \omega_r=0,l_d=1,n_r=0,l_r=0$, initial condition: $h(x)=\frac 23\sum_{k=0}^\infty (\frac{x}{3})^k = \frac{2}{3-x}$)}
\end{figure}

\vspace{1cm}

\begin{figure}[h!]
	\centering
	\begin{overpic}[width=0.9\textwidth]{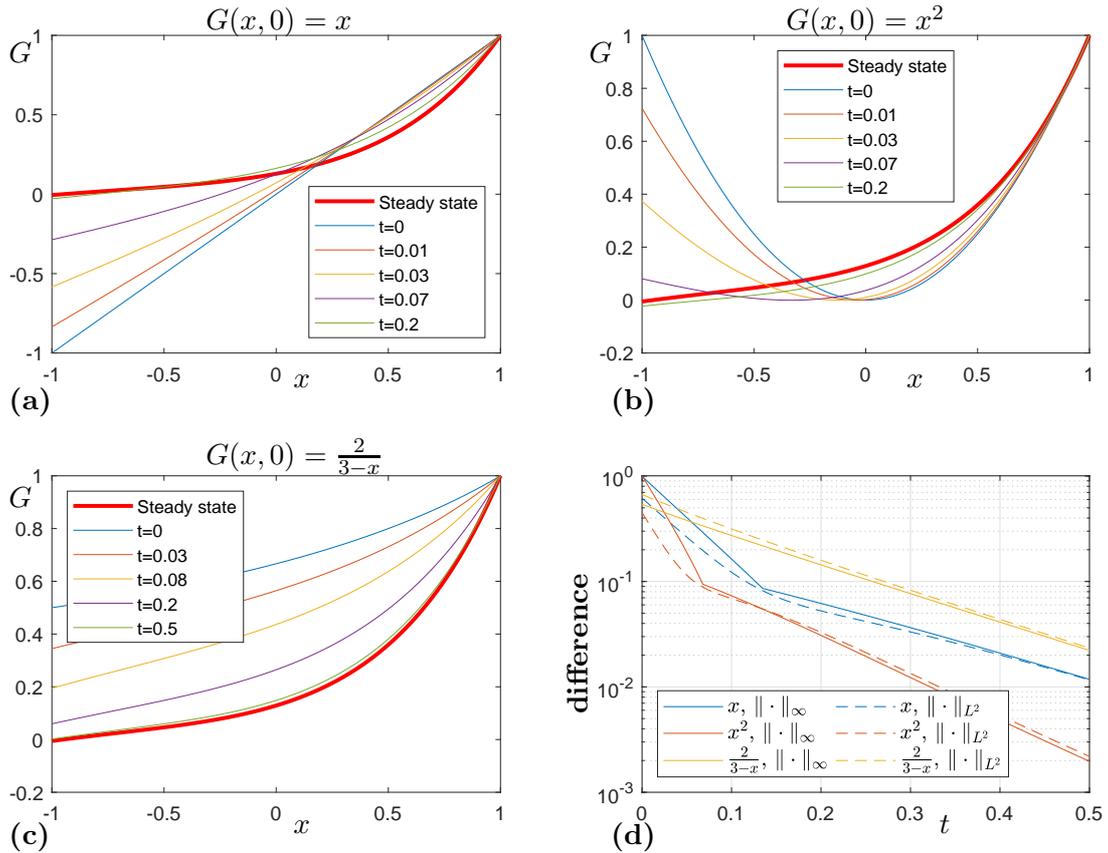}
		\put(0,38){\textbf{(a)}}
		\put(55,38){\textbf{(b)}}
		\put(0,-2){\textbf{(c)}}
		\put(55,-2){\textbf{(d)}}
		\put(26,-1){$x$}
		\put(0,29){$G$}
		\put(85,-1){$t$}
		\put(51,10){\rotatebox{90}{\textbf{difference}}}
		\put(0,70){$G$}
		\put(53,70){$G$}
		\put(26,40){$x$}
		\put(82,40){$x$}
		\put(18,73){$G(x,0)=x$}
		\put(18,33){$G(x,0)=\frac{2}{3-x}$}
		\put(71,73){$G(x,0)=x^2$}
	\end{overpic}
	\captionsetup{aboveskip=18pt}
	\caption{\label{fig:combineinitialvalues} (a) Solution $G(\cdot,t)$ for different values of $t$ with initial condition $G(x,0)=x$. (b) Solution $G(\cdot,t)$ for different values of $t$ with initial condition $G(x,0)=x^2$. (c) Solution $G(\cdot,t)$ for different values of $t$ with initial condition $G(x,0)=\frac{2}{3-x}$. (d) Difference of the solution $G(\cdot,t)$ to the steady state with respect to $\norm\cdot_\infty$ and $\norm\cdot_{L^2}$ for all three initial conditions in dependence of time. (Parameter values: $m=3, n_d=1, \omega_p=0, l_p=0, n_p=0, \omega_r=1, l_d=1, n_r=1, l_r=1$)}
\end{figure}

Finally, we want to investigate the influence of different initial conditions to the rate of convergence. For that purpose, we fix all of our coefficients and vary only the initial conditions. As shown in Figure \ref{fig:combineinitialvalues}(a-c) the solution $G(\cdot ,t)$ converges to the steady state in each case. Surprisingly, even though the parameters are the same for each simulation in Figure \ref{fig:combineinitialvalues}, we can see that the rates of convergence are different. Another interesting fact is that for the initial conditions $G(x,0)=x$ and $G(x,0)=x^2$ the map $t\mapsto \norm{G(\cdot,t)-G^\star}_\infty$ in Figure \ref{fig:combineinitialvalues}(d) has a sharp bend. This can be explained by the fact that for $t$ at the bend, the point where the supremum is attained discontinuously changes from $x=-1$ to $x\in (-1,1)$.

\section{Conclusion}

In this work, we have analyzed a certain class of nonlocal partial differential equations arising in the generation of complex networks. We have shown that the equation is well-posed in the classical sense by using a combination of tools from differential equations. The first step was to convert the nonlocal problem with a point nonlocality into a local problem by solving an auxiliary ordinary differential equation. In the second step, we used the method of characteristics in combination with ideas about generating functions to establish solvability and regularity for the PDE. Furthermore, we studied the existence of steady states analytically, and their stability numerically. For future work, there are still open problems regarding global analytical stability and the convergence speed to steady states. We conjecture, that in most reasonable cases global stability holds and that exponential convergence takes places in $L^p$ or more general Sobolev spaces except for degenerate parameter configurations.\medskip 

In summary, we have contributed several mathematical techniques to analyze a particular class of nonlocal models. Yet, we emphasize that although the model studied here was motivated by one particular work of Silk et al~\cite{Silk2016}, there is a larger trend that nonlocal PDEs appear in the dynamics of and on complex networks. For example, there is the recent theory of PDEs over graphons~\cite{Medvedev3,KuehnThrom2}, where coupling structures are encoded in a nonlocal integro-differential equation. Furthermore, there is extensive work on nonlocal mean-field limits of coupled oscillator models connected to chimera states~\cite{KuramotoBattogtokh,AbramsStrogatz} as well as on nonlocal neural field models derived from networks~\cite{Bressloff,Laing}. This concentration of activity from different scientific communities on nonlocal PDEs in the network science context does not seem to be just accidental but points towards a bigger emerging theme. Therefore, one may expect that nonlocal PDE methods for network analysis are going increase further in their relevance. 

\bibliography{References}

\end{document}